\documentclass{kms-j}

\newcommand{\publname}{}
\issueinfo{}
{}
{}
{}
\pagespan{1}{}
\copyrightinfo{}
{Korean Mathematical Society}

\usepackage{graphicx}
\allowdisplaybreaks

\theoremstyle{plain}
\newtheorem{theorem}{Theorem}[section]

\newtheorem{lemma}[theorem]{Lemma}

\theoremstyle{definition}

\theoremstyle{remark}

\begin{document}

\title[]
{Existence and uniqueness of solutions to Bogomol'nyi-Prased-Sommerfeld equations on graphs}

\author[Y. Hu]{Yuanyang Hu}
\address{Yuanyang Hu \\  School of Mathematics and Statistics \\ Henan University \\ Kaifeng, 475004, P. R. China}
\email{yuanyhu@mail.ustc.edu.cn}


\begin{abstract}
Let $G=(V,E)$ be a connected finite graph.
We investigate two Bogomol'nyi-Prased-Sommerfeld equations on $G$. We establish necessary and sufficient conditions for the existence and uniqueness of solutions to the BPS equations. 
\end{abstract}

\maketitle

\section{Introduction}
Vortices involve momentous roles in multitudinous fields of theoretical physics comprizing quantum Hall effect, condensed-matter physics, electroweak theory, superconductivity theory, optics, and cosmology. Taubes established the multiple vortex static solutions for the Abelian Higgs model for the first time in \cite{T1, T2, JT}. After that, a large number of work related to vortex equations has been accomplished; see, for example, \cite{CI, Ha, T, TY, Y} and the references therein. Recently, Lin and Yang pursued a systematic research
\cite{LY4,LY5} of the multiple vortex equations obtained in \cite{E3, E4, E5, GJY, SY, SY8, SY9, SY1}. They established a great deal of 
sharp existence and uniqueness theorems. The approach they employed include
a priori estimates, monotone iterations, and a degree-theory argument, and constrained
minimization. In \cite{LY}, Yang and Lieb presented a series of sharp existence and uniqueness theorems for the solutions of some non-Abelian vortex equations derived in [5,6], which provide an essential mechanism for linear confinement. Recently, Han \cite{Ha} established the existence of multipe vortex solutions for the BPS equations derived in \cite{S} from the theory of multi-intersection of D-branes. Chen and Yang \cite{CY1} proved two sharp existence theorems for the non-Abelian BPS vortex equations arising in the supersymmetric $U(1)\times SU(N)$ gauge theory. In this paper, we investigate the
existence and uniqueness of the solutions to the two BPS equations in \cite{CY1, Ha} on a connected finite
graph.

 During the recent time, the investigation on the equations on graphs has drawn a following and considerable attention from scholars; see, for example, \cite{ Bdd, GC, GJ, GLY, Hu, Hu1, HWY, HLY, LP, LW, TZZ, WC, WN} and the references therein. Grigor'yan, Lin and Yang \cite{GLY} studied the Kazdan-Warner equation 
 $$\Delta u=c-he^{u}$$
 on graph. Huang, Lin and Yau \cite{HLY} proved the existence of solutions to mean field equations 
 $$
 \Delta u+e^{u}=\rho \delta_{0}
 $$
 and
$$
\Delta u=\lambda e^{u}\left(e^{u}-1\right)+4 \pi \sum_{j=1}^{M} \delta_{p_{j}}
$$
 on graphs. Huang, wang and Yang \cite{HWY} studied the Mean field equation 
 $$
 \Delta u+\rho\left(\frac{h e^{u}}{\int_{M} h e^{u} d x}-\frac{1}{|M|}\right)=4 \pi \sum_{j=1}^{N} \alpha_{q_{j}}\left(\delta_{q_{j}}-\frac{1}{|M|}\right)
 $$
 and the relativistic Ablian Chern-Simons equations 
 $$
 \left\{\begin{array}{l}
 	\Delta u_{1}=\lambda e^{u_{2}}\left(e^{u_{1}}-1\right)+4 \pi \sum_{j=1}^{k_{1}} \alpha_{j} \delta_{p_{j}}, \\
 	\Delta u_{2}=\lambda e^{u_{1}}\left(e^{u_{2}}-1\right)+4 \pi \sum_{j=1}^{k_{2}} \beta_{j} \delta_{q_{j}}
 \end{array}\right.
 $$
 on the finite connected graphs. For more research on Chern-Simons equations on graphs, we refer the readers to \cite{Hu, Hu1, HS, LP} and references therein.

 The paper is organized as follows. In Section 2, we introduce preliminaries and then state our main results. Section 3 is devoted to the proof of Theorem \ref{t1}. In section 4, we give the proof of Theorem \ref{t2}.
\section{ Settings and Main Results}

 Let $G=(V,E)$ be a connected finite graph, where $V$ denotes the vetex set and $E$ denotes the edge set. 
Let $\mu: V \to (0,+\infty)$ be a finite measure, and $|V|$=$ \text{Vol}(V)=\sum \limits_{x \in V} \mu(x)$ be the volume of $V$. Denote the space of real-valued functions on $V$ by $V^{\mathbb{R}}$. For $x,y\in V$, let $xy$ be the edge from $x$ to $y$. We write $y\sim x$ if $xy\in E$.
Let $\omega: V \times V\to[0,\infty)$ be an edge weight function satisfying $\omega_{xy}=\omega_{yx}$ for all $x,y \in V$
  and $\omega_{xy}>0$ iff $x\sim y$. For any function $u: V \to \mathbb{R}$, the Laplacian of $u$ is defined by 
\begin{equation}\label{1}
	\Delta u(x)=\frac{1}{\mu(x)} \sum_{y \sim x} w_{y x}(u(y)-u(x)).
\end{equation}
The gradient  $\nabla$ of function $f$ is defined by a vector 
$$\nabla f (x):=\left(  \left[ f(y)-f(x)\right]  \sqrt{\frac{w_{xy}}{2\mu (x)}} \right)_{y\sim x} .$$
 The associated gradient form reads 
\begin{equation}\label{g}
	\Gamma(u, v)(x)=\frac{1}{2 \mu(x)} \sum_{y \sim x} w_{x y}(u(y)-u(x))(v(y)-v(x)).
\end{equation}
We denote the length of the gradient of $u$ by
\begin{equation*}
|\nabla u|(x)=\sqrt{\Gamma(u,u)(x)}=\left(\frac{1}{2 \mu(x)} \sum_{y \sim x} w_{x y}(u(y)-u(x))^{2}\right)^{1 / 2}.
\end{equation*}
Denote, for $p>0$ and any $
u\in V^{\mathbb{R}}$, an integral of $u^{p}$ on $V$ by $\int \limits_{V} u^{p} d \mu=\sum\limits_{x \in V} \mu(x) u^{p}(x)$.

In this paper, we study two BPS equations on $G$. One is  
\begin{equation}\label{11}
    \Delta u_j= e^{u_j}+ \sum\limits_{i=1}^{l} \mathrm{e}^{u_i}-(l+1) +4 \pi \sum_{s=1}^{N_j}  \delta_{p_{j,s}} ,~j=1,2,\cdots,l
\end{equation}
$l$, $N_j(j=1,\dots,l)$ are positive integers, $p_{j,s}(j=1,\dots,l,s=1,\dots,N_j)$ are arbitrarily chosen distinct vertices on the graph, and $\delta_{p}$ is the Dirac mass at the vertex $p$. 
The other is
\begin{equation}\label{e1}
    \Delta u_{i}=\sum_{j=1}^{N} a_{i j}\left(\mathrm{e}^{u_{j}}-\theta^{2}\right)+4 \pi \sum_{s=1}^{n_{i}} \delta_{p_{i, s}}(x), \quad i=1, \ldots, N,
\end{equation} 
where 
\begin{equation}\label{2.3}
    a_{i j}=\frac{1}{N}\left(\frac{e^{2}}{2}-\frac{g^{2}}{2 N}\right)+\delta_{i j} \frac{g^{2}}{2 N}, \quad i, j=1, \ldots, N,
\end{equation}
$n_{i}(i=1,\dots,N)$ are positive integers, $\theta,e,g$ are constants and $N$ is a positive integer.

Our main results can be stated as following:
\begin{theorem}\label{t1}
	Equations \eqref{11} admits a unique solution if and only if $\max\limits_{1\le j \le l} \{N_j \} < \frac{(l+1)}{4 \pi} |V|$.
\end{theorem}

\begin{theorem}\label{t2}
	Equations \eqref{e1} admits a unique solution if and only if \begin{equation*}
		n_{i}<\frac{g^{2} {\theta}^{2}}{8 \pi N}|V|+\frac{1}{N}\left(1-\frac{1}{N}\left( \frac{g}{e}\right)^{2}\right) n, \quad i=1, \ldots, N,
	\end{equation*}
where \begin{equation}\label{2.6}
    n=\sum_{i=1}^{N}n_{i}.
\end{equation}
\end{theorem}

 As in \cite{GLY}, we define a sobolev space and a norm by 
\begin{equation*}
	H^{1}(V)=W^{1,2}(V)=\left\{u: V \rightarrow \mathbb{R}: \int \limits_{V} \left(|\nabla u|^{2}+u^{2}\right) d \mu<+\infty\right\},
\end{equation*}
and \begin{equation*}
	\|u\|_{H^{1}(V)}=	\|u\|_{W^{1,2}(V)}=\left(\int \limits_{V}\left(|\nabla u|^{2}+u^{2}\right) d \mu\right)^{1 / 2}.
\end{equation*}

We next give the following Sobolev embedding and Poincar\'{e} inequality which will be used later in the paper.
\begin{lemma}\label{21}
	{\rm (\cite[Lemma 5]{GLY})} Let $G=(V,E)$ be a finite graph. The Sobolev space $W^{1,2}(V)$ is precompact. Namely, if ${u_j}$ is bounded in $W^{1,2}(V)$, then there exists some $u \in W^{1,2}(V)$ such that up to a subsequence, $u_j \to u$ in $W^{1,2}(V)$.
\end{lemma}

\begin{lemma}\label{2.2}
	{\rm (\cite[Lemma 6]{GLY})}	Let $G = (V, E)$ be a finite graph. For all functions $u : V \to \mathbb{R}$ with $\int \limits_{V} u d\mu = 0$, there 
	exists some constant $C$ depending only on $G$ such that $\int \limits_{V} u^2 d\mu \le C \int \limits_{V} |\nabla u|^2 d\mu$.
\end{lemma} 

Hereafter, we use boldfaced letters to denote column vectors and $A^{T}$ to denote the transpose of $A$ for any matrix $A$.

\section{The proof of Theorem \ref{t1}}

Since $\int\limits_{V} 4\pi \sum\limits_{s=1}^{N_j}  \delta_{p_{j,s}} d \mu= 4\pi N_j$, $j=1,2,\cdots,l$, by a similar arguments as in the proof of Lemma 2.4 of \cite{Hu1}, we can show that there exists $u_{j}^{0}$ such that 
\begin{equation}\label{2}
	\Delta u_{j}^{0}= 4\pi \sum\limits_{s=1}^{N_j} \delta_{p_{j,s}}- \frac{4\pi N_j}{|V|},~x\in V,~j=1,2,\cdots,l,
\end{equation}
and that $(u_{1}^{0},u_{2}^{0},\dots,u_{l}^{0})^{T}$ is the unique solution of \eqref{2} by up to a constant vector $\mathbf{C}$. 
Suppoe $\mathbf{u}=(u_1,u_2,\cdots,u_l)$ is a solution to \eqref{11}. Set $v_j:=u_j-u_j^{0}$, $j=1,2,\cdots,l$. Then we know that 
\begin{equation}\label{3}
	\Delta v_j= e^{u_j^{0} +v_j}+ \sum\limits_{i=1}^{l} \mathrm{e}^{u^{0}_i + v_i}-(l+1) +\frac{4\pi N_j}{|V|},~j=1,\cdots,l.
\end{equation}

The following lemma gives a necessary condition for $\eqref{11}$ to have a solution.
\begin{lemma}\label{31}
	If equations $\eqref{11}$ admits a solution, then $\max\limits_{1\le j \le l} \{N_j \} < \frac{(l+1)}{4 \pi} |V|$.
\end{lemma}
\begin{proof}
	Integrating \eqref{3} on $V$, we deduce that 
    $$
    0=\int_V \Delta v_j d \mu=\int_V e^{u_j^0+v_j} d \mu+\sum_{i=1}^l \int_V e^{u_i^0+v_i} d \mu-(l+1)|V|+4 \pi v_j,
    $$
    an hence that
\begin{equation}\label{2.1}
    \int_V e^{u_j^0+v_j} d \mu=(l+1)|V|-4 \pi N_j-\sum_{j=1}^l e^{u_j^0+v_j} d\mu .
\end{equation}
    
    This implies that
    $$
\sum_{j=1}^l\int_V e^{u_j^0+v_j} d \mu=l(l+1)|V|-4 \pi \sum_{j=1}^l N_j-l \sum_{i=1}^l\int_V e^{u_i^0+v_i} d \mu,
    $$
    whence
\begin{equation}\label{22}
    \sum_{i=1}^l \int_V e^{u_j^0+v_j} d \mu=l|V|-\frac{4 \pi \sum_{j=1}^l N_j}{l+1} .
\end{equation}
By \eqref{2.1} and \eqref{22}, we see that
	\begin{equation}\label{4}
		\int\limits_{V} \mathrm{e}^{u_{j}^{0}+v_{j}} \mathrm{~d} \mu=|V|-4 \pi N_{j}+\frac{4 \pi}{l+1} \sum_{i=1}^{l} N_{i} =: K_{j}, \quad j=1, \ldots, l.
	\end{equation}
Since $\int_V e^{u_j^0+v_j} d \mu>0$, we conclude that
$$
|V|-4 \pi N_j+\frac{4 \pi \sum_{j=1}^l N_j}{l+1}>0 .
$$
This implies that
$$
|V|+\frac{4 \pi}{l+1} \sum_{j=1}^l N_j>4 \pi N_j \quad \text { for } \quad j=1,2, \cdots, l,
$$
whence
$$
|V|+\frac{4 \pi}{l+1} l \max _{1 \le j \le l} N_j>4 \pi \max _{1 \le j \le l} N_j .
$$
Thus, we see that
$$
\frac{l+1}{4 \pi}|V|>\max _j N_j.
$$

We now finish the proof.
\end{proof}

The following lemma gives a sufficient condition for \eqref{11} to have a solution.
\begin{lemma}\label{32}
	If $\max\limits_{1\le j \le l} \{N_j \} < \frac{(l+1)}{4 \pi} |V|$, then \eqref{11} has a solution.
\end{lemma}
\begin{proof}
	Rewrite \eqref{3}, we have 
	\begin{equation}\label{3.4}
		\Delta \textbf{v} = A \textbf{E}- \boldsymbol{\Phi}, 
	\end{equation}
where 
$$\Delta \textbf{v}= (\Delta v_1,\cdots,\Delta v_{l})^{T},$$ $$\textbf{E}=(e^{u_{1}^{0} +v_1},e^{u_{2}^{0} +v_2},\cdots,e^{u_{l}^{0}+v_l})^{T},$$ $$\boldsymbol{\Phi}=(l+1-\frac{4\pi N_{1}}{|V|},\cdots,l+1-\frac{4\pi N_{l}}{|V|})$$ and
\begin{equation}
	A=(a_{ij})_{l\times l}=\left(\begin{array}{ccccc}
		2 & 1 & 1 & \ldots & 1 \\
		1 & 2 & 1 & \ldots & 1 \\
		1 & 1 & 2 & \ldots & 1 \\
		\vdots & \vdots & \vdots & \ddots & \vdots \\
		1 & 1 & 1 & \ldots & 2
	\end{array}\right).
\end{equation}
In order to find the energy functional of \eqref{3}, we have the following argument.

It is easy to see that $A={B}{B}^{T}$, where $B=(b_{ij})_{l\times l}$ and $$b_{ij}=\begin{cases}
	\sqrt{\frac{1+i}{i} },~i=j, \\
		\sqrt{\frac{1}{j(j+1)} },~i\not= j,~i<j, \\
		0, ~i\not= j,~i>j.
\end{cases}$$  
Let \begin{equation}\label{8}
	\textbf{q}= B^{-1} \textbf{v}, 
\end{equation}where $\textbf{v}=(v_1,\cdots,v_l)^{T}$ and $B^{-1}$ is the inverse of matrix $B$. Direct calculations yield that 
\begin{equation}
    B^{-1}_{l\times l}=\left(\begin{array}{ccccc}
        \frac{\sqrt{2}}{2} & 0 & 0 & \ldots & 0 \\
        -\frac{\sqrt{6}}{6} & \frac{\sqrt{6}}{3} & 0 & \ldots & 0 \\
        \frac{-\sqrt{12}}{12} & -\frac{\sqrt{12}}{12} & \frac{\sqrt{3}}{2} & \ldots & 0 \\
        \vdots & \vdots & \vdots & \ddots & \vdots \\
        -\sqrt{\frac{{1}}{l(l+1)}} & -\sqrt{\frac{{1}}{l(l+1)}} & -\sqrt{\frac{{1}}{l(l+1)}} & \ldots & \sqrt{\frac{{1}}{l(l+1)}}
    \end{array}\right).
\end{equation}
Thus, from \eqref{3.4}, we know that
\begin{equation}\label{3.10}
    \begin{aligned}
        \Delta {\bf q} & =B^{-1} A {\bf E}-B^{-1}{\bf \Phi} \\
        & =B^{-1} B B^{T} {\bf E}-B^{-1}{\bf \Phi} \\
        & =B^{T} {\bf E}-B^{-1}{\bf \Phi}.
    \end{aligned}
\end{equation}
By \eqref{8}, we have 
\begin{equation}\label{3.11}
    \left\{\begin{array}{l}
        v_1=\sqrt{2} q_1 \\
        v_j=\sum_{i=1}^{j-1} \frac{q_i}{\sqrt{i(i+1)}}+\sqrt{\frac{j+1}{j}} q_j, \quad j=2, \ldots, l .
    \end{array}\right.
\end{equation}
Applying \eqref{3.10} and \eqref{3.11}, we have
\begin{equation}\label{5}
	\begin{aligned}
		\Delta q_{1}=& \sqrt{2} \exp \left(u_{1}^{0}+\sqrt{2} q_{1}\right) \\
		&+\frac{1}{\sqrt{2}} \sum_{i=2}^{l} \exp \left(u_{i}^{0}+\sum_{k=1}^{i-1} \frac{q_{k}}{\sqrt{k(k+1)}}+\sqrt{\frac{i+1}{i} }q_{i}\right)-g_{1},
	\end{aligned}
\end{equation} 
\begin{equation}\label{6}
	\begin{aligned}
		\Delta q_{j}=& \sqrt{\frac{j+1}{j}} \exp \left(u_{j}^{0}+\sum_{k=1}^{j-1} \frac{q_{k}}{\sqrt{(k+1)k}}+\sqrt{\frac{j+1}{j} }q_{j}\right) \\
		&+\sqrt{\frac{1}{j(j+1)}} \sum_{i=j+1}^{l} \exp \left(u_{i}^{0}+\sum_{k=1}^{i-1} \frac{q_{k}}{\sqrt{k(k+1)}}+\sqrt{\frac{i+1}{i} }q_{i}\right)-g_{j},\\
		&j=1,\cdots,l,
	\end{aligned}
\end{equation}
\begin{equation}\label{7}
	\Delta q_{l}=\sqrt{\frac{l+1}{l}} \exp \left(u_{l}^{0}+\sum_{k=1}^{l-1} \frac{q_{k}}{\sqrt{(k+1)k}}+\sqrt{\frac{l+1}{l}} q_{l}\right)-g_{l},
\end{equation}
where $\textbf{g}=(g_1,\cdots,g_l)^{T}=B^{-1}\boldsymbol{\Phi}$. 

Now we define the energy functional
\begin{equation}\label{3.15}
	\begin{gathered}
		I(\mathbf{q})=I\left(q_{1}, \ldots, q_{l}\right)=\int\limits_{V}\left\{\frac{1}{2} \sum_{i=1}^{l}   \Gamma({q}_{i},{q}_{i})        +\exp \left(u_{1}^{0}+\sqrt{2} q_{1}\right)-g_{1} q_{1}\right. \\
		\left.+\sum_{i=2}^{l} \exp \left(u_{i}^{0}+\sum_{k=1}^{i-1} \frac{q_{k}}{\sqrt{k(k+1)}}+\sqrt{\frac{i+1}{i} }q_{i}\right)-g_{i} q_{i}\right\} \mathrm{d} \mu .
	\end{gathered}
\end{equation}
It is easy to check that if $I$ has a critical point, then it is a solution of equations \eqref{5}-\eqref{7}. Furthermore, one can check that the system of equations \eqref{5}-\eqref{7} are the Euler Lagrange equations of the functional $I$.

We next prove $I$ has a critical point.

 Define $\mathbf{K}:=(K_1 ,\dots,K_l)^{T}$. By \eqref{4}, we have $\mathbf{K}= |V| (B^{T})^{-1} \mathbf{g} .$ For any $q\in H^{1}(V)$, we can write 
 \begin{equation}\label{no}
 	q=\bar{q} + \hat{q}, 
 \end{equation}where $ \int\limits_{V} \hat{q} d \mu =0 $ and $\bar{q}= \frac{\int\limits_{V} q d\mu}{|V|}$. 
 
 For any $\textbf{q} \in \underbrace{H^{1}(V)\times \dots \times H^{1}(V)}_{l}$, from \eqref{3.11} and \eqref{3.15}, we conclude that
\begin{equation}\label{3.17}
	\begin{aligned}
		I(\mathbf{q}) &=\int\limits_{V} \frac{1}{2} \sum_{i=1}^{l}  \Gamma(\hat{q}_{i},\hat{q}_{i})    \mathrm{d} \mu+\int\limits_{V} \sum_{i=1}^{l} \exp \left(u^{0}_{i}+\hat{v}_{i}+\bar{v}_{i}\right) \mathrm{d} \mu-K^{T} \bar{\mathbf{v}} \\
		&=\int\limits_{V} \frac{1}{2} \sum_{i=1}^{l}  \Gamma(\hat{q}_{i},\hat{q}_{i})    \mathrm{d} \mu +\int\limits_{V} \sum_{i=1}^{l} \exp \left(u^{0}_{i}+\hat{v}_{i}+\bar{v}_{i}\right) \mathrm{d} \mu-\sum_{i=1}^{l} K_{i} \bar{v}_{i}.
	\end{aligned}
\end{equation}
By Jensen's inequality, we obtain 
\begin{equation}
exp\left( \frac{\int\limits_V u_{i}^{0}+\hat{v}_{i}+\bar{v}_{i} d\mu}{|V|}	 \right)  \le \frac{1}{|V|} \int\limits_V exp\left(u_{i}^{0}+\hat{v}_{i}+\bar{v}_{i}\right) d\mu.
\end{equation}
Thus we deduce that
\begin{equation}
	\int\limits_V exp\left(u_{i}^{0}+\hat{v}_{i}+\bar{v}_{i}\right) d\mu \ge |V|exp\left( \frac{1}{|V|} \int\limits_{V} u_{i}^{0} d\mu \right) e^{\bar{v}_{i}} =:c_{i} e^{\bar{v}_{i}}~\text{for}~i=1,\cdots,l.
\end{equation}
From \eqref{4}, we obtain $K_i >0$ for $i=1,\cdots,l$. Hence, applying the elementary inequality 
\begin{equation}\label{320}
    \frac{a}{b}\left(1-\ln \left[\frac{a}{b c}\right]\right) \leq c e^{b x}-a x, \quad a, b, c>0, \quad x \in \mathbb{R}
\end{equation}
in \eqref{3.17}, we have
\begin{equation}\label{9}
	\begin{aligned}
	I(\mathbf{q}) &\ge \int_{V} \frac{1}{2} \sum\limits_{i=1}^{l} \Gamma(\hat{q}_{i},\hat{q}_{i}) d\mu+ \sum\limits_{i=1}^{l}(c_{i}e^{\bar{v}_{i}}- K_{i} \bar{v}_{i}) \\
	&\ge \int_{V} \frac{1}{2} \sum\limits_{i=1}^{l} \Gamma(\hat{q}_{i},\hat{q}_{i}) d \mu + \sum\limits_{j=1}^{l} K_{j} (ln{\frac{c_j}{K_j}}+1).
\end{aligned}
\end{equation}
It follows that $I$ is bounded from below in $H^{1}(V)$. Furthermore, it is easy to check that $I$ is strictly convex.
We next show that $I$ is weakly lower semi-continuous in $H^{1}(V)$. 
Suppose that $\{q_{1}^{(k)},q_{2}^{(k)},\cdots,q_{l}^{(k)}\}_{k=1}^{\infty}$ satisfying $(q_{1}^{(k)},q_{2}^{(k)},\cdots,q_{l}^{(k)})\rightharpoonup(q_{1},q_{2},\cdots,q_{l})$ in 
$\underbrace{H^{1}(V)\times \dots \times H^{1}(V)}_{l}$, i.e., $q_{i}^{(k)}\rightharpoonup q^{k}$ in $H^{1}(V)$ for all $i=1,2,\cdots,l$. Since $V$ is a finite graph, $H^{1}(V)=V^{\mathbb{R}}=L^{2}(V)$. Hence, we know that the dual space to $H^{1}(V)$ is $V^{\mathbb{R}}=L^{2}(V)$. This implies that 
\begin{equation}\label{3.22}
    \sum_{x\in V}q_{i}^{(k)}(x)f(x)\mu(x)\to\sum_{x\in V}q_{i}(x)f(x)\mu(x)~\text{ as }~k\to\infty
\end{equation} 
for all $f\in V^{\mathbb{R}}$ and $i=1,2,\cdots,l$. Fix $x_0\in V$. Taking 
\begin{equation}
   f(x)=\left\{\begin{aligned}
       \frac{1}{\mu(x_0)},~~x=x_0,\\
       0,~~x\not=x_0.
   \end{aligned}\right. 
\end{equation}
in \eqref{3.22}. Then we see that
\begin{equation}
q_{i}^{(k)}(x)\to q(x)~\text{as}~k\to\infty~\text{ uniformly for all } x\in V \text{and}~ i=1,2,\cdots,l. 
\end{equation}
This implies that 
\begin{equation}
    I(\bf{q})\le \liminf_{k\to \infty} I(\bf{q}^{k}),
\end{equation}
where ${\bf{q}}=(q_{1},q_{2},\cdots,q_{l})^{T}$, ${\bf q}^{(k)}=(q_{1}^{(k)},q_{2}^{(k)},\cdots,q_{l}^{(k)})^{T}$.
Thus we can choose a minimizing sequence $\{(q_{1,k},\dots,q_{l,k})\}_{k=1}^{\infty}$ of the following minimization problem
$$\inf\limits\{I(\mathbf{q}) | \mathbf{q}=(q_1,\dots,q_l)^{T}\in \underbrace{H^{1}(V)\times \dots \times H^{1}(V)}_{l}\} .$$ 
In view of $\lim\limits_{t \to \pm \infty} c_{i} e^{t}- K_{i}t=\infty$ for $i=1,\dots,l$. we deduce from \eqref{9} that $\{\bar{v}_{i,k}\}_{k=1}^{\infty}$ is bounded for $i=1,\dots,l$. By \eqref{8}, we know that $\{\bar{q}_{i,k}\}_{k=1}^{\infty}$ is bounded for $i=1,\cdots,l$. From \eqref{9}, we see that $\{ |\nabla \hat{q}_{i,k}| \}_{k=1}^{\infty}$ is bounded in $L^{2} (V)$ for $i=1,\cdots,l.$ From Lemma \ref{2.2}, we conclude that $\{ \hat{q}_{i,k} \}_{k=1}^{\infty}$ is bounded in $L^{2}(V)$, $i=1,\cdots,l.$ Thus, $\{q_{i,k} \}_{k=1}^{\infty}$ is bounded in $L^{2}(V)$. Therefore, $\{q_{i,k} \}_{k=1}^{\infty}$ is bounded in $H^{1} (V).$ Therefore, by Lemma \ref{21}, there exists $\textbf{q}_{\infty}:=\left( q_{1,\infty},\dots,q_{l,\infty} \right)^{T} \in \underbrace{H^{1}(V)\times \dots \times H^{1}(V)}_{l} $ so that, by passing to a subsequent, 
\begin{equation}
    q_{i,k} \to q_{i,\infty}~in~H^{1}(V) \text{ as } k\to +\infty \text{ for } i=1,\cdots,l, 
\end{equation}
and
\begin{equation}
	q_{i,k} \to q_{i,\infty} \text{ uniformly for } x\in V \text{ as } k\to +\infty \text{ for } i=1,\cdots,l. 
\end{equation}
Thus, $\mathbf{q}_{\infty}$ is a critical point of $I$. Since $I$ is srtictly convex in $H^{1}(V)$, we know that the solution of equations \eqref{5}-\eqref{7} is unique.

We now complete the proof.
\end{proof}

\begin{proof}[Proof of Theorem \ref{t1}]
	The desired conclusion follows directly from Lemmas \ref{31} and \ref{32}.
\end{proof}

\section{The proof of Theorem \ref{t2}}

	Set $u_{i}^{0}$ be a solution of
	\begin{equation}\label{23}
		\Delta u_{i}^{0}=-\frac{4 \pi n_{i}}{|V|}+4 \pi \sum_{s=1}^{n_{i}} \delta_{p_{i, s}}(x), \quad i=1, \ldots, N.
	\end{equation}
Set \begin{equation}
	u_{i}=u_{i}^{0}+U_{i}, \quad i=1, \ldots, N.
\end{equation}
Then \eqref{e1} is transformed into 
\begin{equation}\label{24}
	\Delta U_{i}=\sum_{j=1}^{N} a_{i j}\left(\mathrm{e}^{u_{j}^{0}+U_{j}}-\theta^{2}\right)+\frac{4 \pi n_{i}}{|V|}, \quad i=1, \ldots, N.
\end{equation}
We now write \eqref{24} in the vector form
\begin{equation}\label{25}
	\Delta \mathbf{U}=H\mathbf{G}+\mathbf{F},
\end{equation} 
where 
\begin{equation}
	\mathbf{U}=\left(U_{1}, \ldots, U_{N}\right)^{T}, \quad \mathbf{G}=\left(\mathrm{e}^{u_{1}^{0}+U_{1}}, \ldots, \mathrm{e}^{u_{N}^{0}+U_{N}}\right)^{T},
\end{equation}
\begin{equation}
	\mathbf{F}=\left(\frac{4 \pi n_{1}}{|V|}-\theta^{2} \sum_{j=1}^{N} a_{1 j}, \ldots, \frac{4 \pi n_{N}}{|V|}-\theta^{2} \sum_{j=1}^{N} a_{N j}\right)^{T} =:\left(f_{1}, \ldots, f_{N}\right)^{T},
\end{equation}
and $H=\left(a_{i j}\right)_{N \times N}$ is a $N\times N$ matrix. In view of \eqref{2.3}, one may check that $H$ is positive definite. Then, applying the Cholesky decomposition theorem, we can write 
\begin{equation}
	H=S^{T} S,
\end{equation}
where $S=(t_{ij})_{N\times N}$ is an upper triangular matrix,  
$$
\begin{aligned}
	&t_{11}=\sqrt{a+b}, \quad t_{12}=t_{13}=\cdots=t_{1 N}=\frac{a}{t_{11}} =: \alpha_{1}>0, \\
	&t_{22}=\sqrt{(a+b)-\alpha_{1}^{2}}, \quad t_{23}=t_{24}=\cdots=t_{2 N}=\frac{a-\alpha_{1}^{2}}{t_{22}}=: \alpha_{2}>0, \\
	&\ldots
\end{aligned}
$$

$$
t_{N-1,N-1}=\sqrt{(a+b)-\sum_{i=1}^{N-2} \alpha_{i}^{2}}, \quad t_{N-1,N}=\frac{a-\sum_{i=1}^{N-2} \alpha_{i}^{2}}{t_{N-1,N-1}}=: \alpha_{N-1}>0,
$$
$$
t_{N,N}=\sqrt{(a+b)-\sum_{i=1}^{N-1} \alpha_{i}^{2}} ,
$$
 $a=\frac{\left(e^{2} / 2-g^{2} / 2 N\right)}{N}$ and $b=\frac{g^{2}}{2N}$.
Set $\mathbf{v}=\left(v_{1}, \ldots, v_{N}\right)^{T}, L=\left(S^{T}\right)^{-1} =: \left(l_{i j}\right)_{N \times N}$ and $$\mathbf{v}=L \mathbf{U}.$$
Combining this with \eqref{25}, one may obtain
\begin{equation}\label{49}
	\Delta \mathbf{v}=S\mathbf{G}+L \mathbf{F}.
\end{equation}
Taking $\alpha_{N}=0$, we rewrite \eqref{49} as  
\begin{equation}\label{4.9}
	\Delta v_{i}=t_{i i} \mathrm{e}^{u_{i}^{0}+t_{i i} v_{i}+\sum_{k=1}^{i-1} \alpha_{k} v_{k}}+\alpha_{i} \sum_{j=i+1}^{N} \mathrm{e}^{u_{j}^{0}+t_{j j} v_{j}+\sum_{k=1}^{j-1} \alpha_{k} v_{k}}+\sum_{j=1}^{i} l_{i j} f_{j}, \quad i=1, \ldots, N.
\end{equation}
Here, we understand $\sum_{j=N+1}^{N}$ as $\sum_{j=N}^{N}$ when $i=N$.
Define the energy functional 
\begin{equation}\label{410}
	J(\mathbf{v})=\int\limits_{V}\left\{\frac{1}{2} \sum_{i=1}^{N}  \Gamma(v_{i},v_{i}) +\sum_{i=1}^{N} \mathrm{e}^{u_{i}^{0}+t_{i i} v_{i}+\sum_{k=1}^{i-1} \alpha_{k} v_{k}}+\sum_{i=1}^{N}\left(\sum_{j=1}^{i} l_{i j} f_{j}\right) v_{i}\right\} \mathrm{d} \mu .
\end{equation}
Then it is easy to check that the system of equations \eqref{4.9} are the Euler Lagrange equations of the functional $J$.
The following lemma gives a necessary condition for \eqref{e1} to admit a solution.
\begin{lemma}\label{311}
	If \eqref{e1} admits a solution, then 
	\begin{equation}
		n_{i}<\frac{g^{2} \theta^{2}}{8 \pi N}|V|+\frac{1}{N}\left(1-\frac{1}{N}\left[\frac{g}{e}\right]^{2}\right) n, \quad i=1, \ldots, N.
	\end{equation}
\end{lemma}
\begin{proof}
	Let \begin{equation}
		q_{i}=\int_{V} \mathrm{e}^{u_{i}^{0}+t_{i i} v_{i}+\sum_{k=1}^{i-1} \alpha_{k} v_{k}} \mathrm{~d} \mu, \quad i=1, \ldots, N.
	\end{equation}
From \eqref{4.9}, we get
\begin{equation}\label{413}
	t_{i i} q_{i}+\alpha_{i} \sum_{j=i+1}^{N} q_{j}=-|V| \sum_{j=1}^{i} l_{i j} f_{j} =: p_{i}, \quad i=1, \ldots, N.
\end{equation}
By \eqref{413}, we deduce that
\begin{equation}\label{14}
S\mathbf{q}=-|V| L \mathbf{F}=\mathbf{p},
\end{equation}
where $\mathbf{p}=(p_1,\dots,p_N)^{T}$ and $\mathbf{q}=(q_1,\dots,q_N)^{T}$.
This implies that \begin{equation}\label{4.15}
   \mathbf{q}=-|V|S^{-1}(S^{T})^{-1}F=-|V|H^{-1}F.
\end{equation}
Direct calculations yield that
\begin{equation}\label{15}
H^{-1}=\frac{1}{b(N a+b)}\left(\begin{array}{cccc}
		(N-1) a+b & -a & \cdots & -a \\
		-a & (N-1) a+b & \cdots & -a \\
		\cdots & \cdots & \cdots & \cdots \\
		-a & -a & \cdots & (N-1) a+b
	\end{array}\right).
\end{equation}
From \eqref{15} and \eqref{4.15}, we deduce that
\begin{equation*}
		q_{i} 
		=\theta^{2}|V|+\frac{4 \pi a}{b(N a+b)} \sum_{j=1}^{N} n_i-\frac{4 \pi}{b} n_{i}, \quad i=1, \ldots, N.
\end{equation*}
Recalling that $a=\frac{\left(e^{2} / 2-g^{2} / 2 N\right)}{N}$ and $b=\frac{g^{2}}{2N}$, it follows that
\begin{equation}
	q_{i}=\theta^{2}|V|+8 \pi\left(\frac{1}{g^{2}}-\frac{1}{N e^{2}}\right) n-\frac{8 \pi N}{g^{2}} n_{i}>0, \quad i=1, \ldots, N.
\end{equation}

We now complete the proof.
\end{proof}

We give a sufficient condition for equations \eqref{24} to have a solution in the following lemma.
\begin{lemma}\label{41}
	If \begin{equation}
		n_{i}<\frac{g^{2} \theta^{2}}{8 \pi N}|V|+\frac{1}{N}\left(1-\frac{1}{N}\left(\frac{g}{e}\right)^{2}\right) n, \quad i=1, \ldots, N,
	\end{equation} then \eqref{24} admits a solution.
\end{lemma}
\begin{proof}
For any $\mathbf{v}=(v_1,\dots,v_N)^{T}\in \underbrace{H^{1}(V)\times \dots \times H^{1}(V)}_{N}$. By the notation \eqref{no} and Jensen's inequality, we have
\begin{equation}\label{4.19}
	\begin{aligned}
		\int_{V} \mathrm{e}^{u_{i}^{0}+t_{i i}\left(\bar{v}_{i}+\hat{v}_{i}\right)+\sum_{k=1}^{i-1} \alpha_{k}\left(\bar{v}_{k}+\hat{v}_{k}\right)} \mathrm{d} \mu & \geq|V| \exp \left(\frac{1}{|V|} \int_{V} u_{i}^{0} \mathrm{~d} \mu \right) \exp \left(t_{i i} \bar{{v}}_{i}+\sum_{k=1}^{i-1} \alpha_{k} \bar{{v}}_{k}\right) \\
		& =: C_{i} \mathrm{e}^{t_{i i} \bar{v}_{i}+\sum_{k=1}^{i-1} \alpha_{k} \bar{v}_{k}}, \quad i=1, \ldots, N.
	\end{aligned}
\end{equation}
By \eqref{413}, we rewrite \eqref{410} as 
\begin{equation}\label{420}
    J(\mathbf{v})=\int_{V}\left(\frac{1}{2} \sum_{i=1}^N\Gamma (\hat{v}_{i},\hat{v}_{i}) +\sum_{i=1}^N \mathrm{e}^{u_i^0+t_{i i} v_i+\sum_{k=1}^{i-1} \alpha_k v_k}\right) \mathrm{d}\mu-\sum_{i=1}^N p_i \bar{v}_i.
\end{equation}
Combining \eqref{420} with \eqref{4.19}, we know that 
\begin{equation}
    J(\mathbf{v})-\frac{1}{2} \int_{V} \sum_{i=1}^N\Gamma (\hat{v}_{i},\hat{v}_{i})  \mathrm{~d}\mu \geq \sum_{i=1}^N C_i \mathrm{e}^{t_{ii} \bar{v}_i+\sum_{k=1}^{i-1} \alpha_k \bar{v}_k}-\sum_{i=1}^N p_i \bar{v}_i .
\end{equation}
From \eqref{14}, one may obtain 
\begin{equation}
    p_i=t_{i i} q_i+\alpha_i \sum_{j=i+1}^N q_j, \quad i=1, \ldots, N.
\end{equation}
It follows that 
\begin{equation}
    \sum_{i=1}^N p_i \bar{v}_i=\sum_{i=1}^N q_i\left(t_{i i} \bar{v}_i+\sum_{k=1}^{i-1} \alpha_k \bar{v}_k\right) .
\end{equation}
Define \begin{equation}\label{s2}
	\bar{w}_{i}:=t_{i i} \bar{v}_{i}+\sum_{k=1}^{i-1} \alpha_{k} \bar{v}_{k}, \quad i=1, \ldots, N.
\end{equation} 
Hence by inequality \eqref{320}, we conclude that 
\begin{equation}\label{4.25}
\begin{aligned}
J(\mathbf{v})-\frac{1}{2} \int\limits_{V} \sum_{i=1}^{N}\left|\nabla \hat{v}_{i}\right|^{2} \mathrm{~d} \mu
 &\geq \sum_{i=1}^{N}\left(C_{i} \mathrm{e}^{\bar{w}_{i}}-q_{i} \bar{w}_{i}\right) \\
 & \ge \sum_{i=1}^{N} q_{i}\left(1+\ln \left(\frac{C_{i}}{q_{i}}\right)\right).		
\end{aligned}
\end{equation}
Considering the following minimization problem 
\begin{equation}\label{m}
	\eta \equiv \inf \left\{J(\mathbf{v}) \mid \mathbf{v} \in \underbrace{H^{1}(V)\times \dots \times H^{1}(V)}_{N}\right\}.
\end{equation}
By a similar argument as in the proof of Lemma \ref{32}, we see that $J$ is weakly lower semi-continuous in $H^{1}(V)$.
Let $\{ (v_{1,k},\dots,v_{N,k})\}_{k=1}^{\infty}$ be a minimizing sequence of \eqref{m}, 
In view of \begin{equation}
 \lim_{t\to\infty}C_{i}e^{t}-q_{i}t=\infty~\text{ for }~i=1,\cdots,N.
\end{equation}
Using \eqref{4.25}, one may deduce that
\begin{equation}
\{\bar{w}_{i,k}\}_{k=1}^{\infty} \text{ is bounded for }i=1,\cdots,N,
\end{equation} where \begin{equation}\label{s2}
\bar{w}_{i,k}:=t_{i i} \bar{v}_{i,k}+\sum_{j=1}^{i-1} \alpha_{j} \bar{v}_{j,k}, \quad i=1, \ldots, N.
\end{equation}
Combining this with \eqref{s2}, 
\begin{equation}\label{429}
    \{\bar{v}_{i,k}\}_{k=1}^{\infty} \text{ is bounded for }i=1,\cdots,N.
\end{equation}
From \eqref{4.25}, we see that $\{\|\nabla\hat{v}_{i,k}\|_{2}\}_{k=1}^{\infty}$ is bounded for all $i=1,2,\cdots,N$. Then, by Lemma \ref{2.2}, one may obtain 
 $\{\|\hat{v}_{i,k}\|_{2}\}_{k=1}^{\infty}$ is bounded for all $i=1,2,\cdots,N$. From this and \eqref{429}, $\{{v}_{i,k}\}_{k=1}^{\infty}=\{\bar{v}_{i,k}+\hat{v}_{i,k}\}_{i=1}^{\infty}$ is bounded in $H^{1}(V)$ for all $i=1,2,\cdots,N$. Therefore,
we can deduce that there exists $\mathbf{v_{\infty}}:=(v_{1,\infty},\dots,v_{N,\infty})^{T}$ such that, by passing to a subsequence, 
\begin{equation}
	v_{i,k}\to v_{i,\infty}
\end{equation}
uniformly for $x\in V$ as $k\to+\infty$ for $i=1,\dots,N.$ Thus, $\mathbf{v}_{\infty}$ is a critical point of $J$. It's easy to check that $J$ is strictly convex in $H^{1}(V)$. Thus, we know that the solution of equations \eqref{4.9} is unique.

The proof is finished.
\end{proof}

\begin{proof}[Proof of Theorem \ref{t2}]
	The desired conclusion follows directly from Lemmas \ref{311} and \ref{41}.
\end{proof}

Next, we give a constrained minimization approach to the problem.

Denote \begin{equation}\label{s3}
	I_{i}(\mathbf{v})=\int\limits_{V} \mathrm{e}^{u_{i}^{0}+t_{i i} v_{i}+\sum_{k=1}^{i-1} \alpha_{k} v_{k}} \mathrm{~d} \mu=q_{i}, \quad i=1, \ldots, N,
\end{equation}
We consider the following constrained minimization problem
\begin{equation}\label{so}
	\gamma= \inf \left\{J(\mathbf{v}) \mid \mathbf{v} \in \underbrace{H^{1}(V)\times \dots \times H^{1}(V)}_{N}, I_{1}(\mathbf{v})=q_{1}, \ldots, I_{N}(\mathbf{v})=q_{N}\right\}.
\end{equation}
We now investigate whether the constraints in \eqref{so} give rise to the so-called "constraints" problem due to the issue of the Lagrange multipliers. For this purpose, let $\mathbf{v}=(v_1,\dots,v_N)^{T}$ be a critical point of $J$ subject to the constraints
\begin{equation}\label{4.33}
	I_{i}(\mathbf{v})=q_{i},~i=1,\dots,N.
\end{equation}
Then we can find real numbers $\sigma_{1},\dots,\sigma_{\textsc{N}}$ such that
\begin{equation}\label{L}
	d_{i} J=\sum\limits_{j=1}^{N} \sigma_j d_{i} I_{j},~i=1,\dots,N,
\end{equation}
where $d_{i}(i=1,\dots,N)$ denote the Fr{$\acute{e}$}chet differention with respect to the i-th arguments, respectively. Let $F=(t_{ij})$, $l_{ij}$ be the entries of the matrix $L=(F^{T})^{-1}$.
Then, for any $z_{1},\dots,z_{N} \in H^{1}(V),$ 
\begin{equation}\label{s1}
	\begin{aligned}
		&\int\limits_{V} \left\{\Gamma( v_{i},z_{i})  +\left(t_{i i} \mathrm{e}^{u_{i}^{0}+t_{i i} v_{i}+\sum\limits_{k=1}^{i-1} \alpha_{k} v_{k}}+\alpha_{i} \sum_{j=i+1}^{N} \mathrm{e}^{u_{j}^{0}+t_{j j} v_{j}+\sum\limits_{k=1}^{j-1} \alpha_{k} v_{k}}+\sum_{j=1}^{i} l_{i j} f_{j}\right) z_{i}\right\} \mathrm{d} \mu \\
		&=\sigma_{i} t_{i i} \int\limits_{V} \mathrm{e}^{u_{i}^{0}+t_{i i} v_{i}+\sum\limits_{k=1}^{i-1} \alpha_{k} v_{k}} z_{i} \mathrm{~d} \mu+\alpha_{i} \sum\limits_{j=i+1}^{N} \sigma_{j} \int\limits_{V} \mathrm{e}^{u_{j}^{0}+t_{j j} v_{j}+\sum\limits_{k=1}^{j-1} \alpha_{k} v_{k}} z_{i} \mathrm{~d} \mu.
	\end{aligned}
\end{equation}
Taking $z_{1},\dots,z_{N}=1$ in \eqref{s1}, we deduce that 
\begin{equation}
	\sigma_{i} t_{i i} q_{i}+\alpha_{i} \sum_{j=i+1}^{N} \sigma_{j} q_{j}=0, \quad i=1, \ldots, N,
\end{equation}
and hence that
\begin{equation}
\sigma_{N}=\sigma_{N-1}=\dots=\sigma_{1}=0,
\end{equation}
which reveals that all terms in \eqref{L} arising from the Lagrange multipliers are automatically absent. Thus, a solution of \eqref{so} satisfies \eqref{4.9}. Applying the notation \eqref{no}, we rewrite 
\eqref{s3} as 
\begin{equation}
	\mathrm{e}^{t_{i i} \bar{v}_{i}+\sum\limits_{k=1}^{i-1} \alpha_{k} \bar{v}_{k}} \int\limits_{V} \mathrm{e}^{u_{i}^{0}+t_{i i} \hat{v}_{i}+\sum\limits_{k=1}^{i-1} \alpha_{k} \hat{v}_{k}} \mathrm{~d} \mu=q_{i}, \quad i=1, \ldots, N,
\end{equation}
from which we deduce that 
\begin{equation}\label{s35}
	\bar{v}_{i}=\sum_{j=1}^{i} l_{i j}\left(\ln q_{j}-\ln I_{j}(\hat{\mathbf{v}})\right), \quad i=1, \ldots, N,
\end{equation}
where 
$$
\hat{\mathbf{v}}=\left(\hat{v}_{1}, \ldots, \hat{v}_{N}\right)^{T}.
$$
Thus, using \eqref{s35}, we can rewrite \eqref{410} as 
\begin{equation}
	\begin{aligned}
		J(\mathbf{v})- \sum_{i=1}^{N} \int\limits_{V} \frac{1}{2}\Gamma (\hat{v}_{i},\hat{v}_{i}) \mathrm{~d} \mu &=\sum_{i=1}^{N} q_{i}-\sum_{i=1}^{N} p_{i} \bar{v}_{i} \\
		&=\sum_{i=1}^{N} \sum_{j=1}^{i} p_{i} l_{i j} \ln I_{j}(\hat{\mathbf{v}})+\sum_{i=1}^{N}\left(q_{i}-p_{i} \sum_{j=1}^{i} l_{i j} \ln q_{j}\right).
	\end{aligned}
\end{equation}
We rewrite \eqref{14} as
\begin{equation}
	J(\mathbf{v})=\frac{1}{2} \sum_{i=1}^{N} \int\limits_{V} \Gamma (\hat{v}_{i},\hat{v}_{i}) \mathrm{~d} \mu+\sum_{i=1}^{N} q_{i} \ln I_{i}(\hat{\mathrm{v}})-C,
\end{equation}
where $C=C(L,\mathbf{p},\mathbf{q})$ ia a constant.
By the Jensen inequality, we deduce that
\begin{equation}
	I_{i}(\hat{\mathbf{v}}) \geq|V| \exp \left(\int_{V} u_{i}^{0} \mathrm{~d} \mu\right) =:\mu_{i}, \quad i=1, \ldots, N.
\end{equation}
Since $q_{i}>0$ for $i=1,2,\cdots,N$, we obtain 
\begin{equation}\label{439}
	J(\mathbf{v})-\frac{1}{2} \sum_{i=1}^{N} \int\limits_{V}\left|\nabla \hat{v}_{i}\right|^{2} \mathrm{~d} \mu \geq \sum_{i=1}^{N} q_{i} \ln \mu_{i}-C.
\end{equation}
Set $\{ (v_{1,k},\dots,v_{N,k}) \}_{k=1}^{\infty}$ be a minimizing sequence of \eqref{so}. By \eqref{439},  $$\{( |\nabla\hat{v}_{1,k}|,\dots,|\nabla\hat{v}_{N,k}|) \}_{k=1}^{\infty}\text{ is bounded in }~L^{2}(V).$$ By Lemma \ref{2.2}, $\{( \hat{v}_{1,k},\dots,\hat{v}_{N,k})\}_{k=1}^{\infty}$ is bounded in $L^{2}(V)$.  By Lemma \ref{21}, we deduce that, by passing to a subsequence,
\begin{equation}
\hat{v}_{i,k}\to w_{i,\infty} \text{ uniformly for } x\in ~V \text{ as } k\to+\infty \text{ for } i=1,\dots,N.
\end{equation}
 Thus, from \eqref{s35}, by passing to a subsequence, 
 \begin{equation}
\bar{v}_{i,k}\to \bar{w}_{i,\infty} \text{ as }k\to+\infty \text{ for all } i=1,2,\cdots,N. 
 \end{equation}
Hence \begin{equation}
    v_{i,k}=\bar{v}_{i,k}+\hat{v}_{i,k}\to\bar{w}_{i,\infty}+\hat{w}_{i,\infty}:=w_{i,\infty}\text{ as }k\to+\infty,~i=1,2,\cdots,N.
\end{equation}
From \eqref{4.33}, \begin{equation}
  I_{i}({\bf v}^{(k)})=q_{i},~i=1,2,\cdots,N.
\end{equation} where ${\bf v}^{(k)}:=({v}_{1,k},{v}_{2,k},\cdots,{v}_{N,k})^{T}$.
Letting $k\to\infty$ in above equality, we deduce that 
\begin{equation}
    I_{i}({\bf w_{\infty}})={q_i},
\end{equation}
where $\mathbf{w}_{\infty}:=(w_{1,\infty},\dots,w_{N,\infty})^{T}$.
Thus, by \eqref{s3}, we know that 
$$\gamma=J(\mathbf{w}_{\infty}),$$
 where $\mathbf{w}_{\infty}:=(w_{1,\infty},\dots,w_{N,\infty})$. 
 Therefore, we know that $\mathbf{w}_{\infty}$ is a solution to the problem \eqref{so}. It follows that $\mathbf{w}_{\infty}$ is a solution to the problem \eqref{4.9}. 

\section* {Acknowledgements}
The author thanks the unknown referee very much for helpful suggestions. This work is financially supported by the China Postdoctoral Science Foundation (Grant No. 2022M711045), and the National
Natural Science Foundation of China (Grant No. 12201184).


\begin{thebibliography}{99}

\bibitem{B} E. B. Bogomol'nyi, {\it The stability of classical solutions}, Soviet J. Nuclear Phys. 24 (1976), 449-454.


\bibitem{Bdd}   D. Bazeia, E. da Hora, C. dos Santos, R. Menezes,  {\it Generalized self-dual Chern–Simons
    vortices}, Phys. Rev. D 81 (2010), 125014.

\bibitem{CI} D. Chae, O. Y. Imanuvilov, {\it Non-topological solutions in the generalized self-dual Chern-
    Simons-Higgs theory}, Calc. Var. Partial Differential Equations 16 (2003), 47-61.


\bibitem{CY}  L. Caffarelli, Y. Yang, {\it Vortex condensation in the Chern–Simons Higgs model: an existence
    theorem}, Comm. Math. Phys. 168 (1995), 321–336.

\bibitem{CY1} S. Chen, Y. Yang, {\it Existence of multiple vortices in supersymmetric gauge field theory}, Proceedings of the Royal Society A: Mathematical, Physical and Engineering Sciences 468(2148) (2012), 3923-3946.

\bibitem{E3} M. Eto, T. Fujimori, T. Nagashima,  M. Nitta, K. Ohashi, N. Sakai, {\it Multiple layer structure of non-Abelian vortex}, Phys. Lett. B 678 (2009), 254–258. 

\bibitem{E4} M. Eto, Y. Isozumi, M. Nitta, K. Ohashi, N. Sakai,  {\it Solitons in the Higgs phase – the moduli matrix
approach}, J. Phys. A 39(26) (2006), R315.

\bibitem{E5} M. Eto, Y. Isozumi,  M. Nitta,  K. Ohashi,   N. Sakai, {\it Moduli space of non-Abelian vortices}, Phys. Rev. Lett. 96(16)  (2006), 161601.

 \bibitem{GHJ} H. Ge, B. Hua,  W. Jiang, {\it A note on Liouville type equations on graphs}, Proc. Amer. Math. Soc. 146(11) (2018), 4837-4842.

\bibitem{GC} H. Ge, W. Jiang, {\it The 1-Yamabe equation on graphs}, Commun. Contemp. Math. 21(08) (2019), 1850040.

\bibitem{GJ} H. Ge, W. Jiang,{ \it Kazdan-Warner equation on infinite graphs}, Journal of the Korean Mathematical Society 55 (2018), 1091–1101.

\bibitem{GJY} S.B. Gudnason,  Y. Jiang, K. Konishi, {\it Non-Abelian vortex dynamics: effective world-sheet action},. J. High Energy Phys. 2010(8) (2010), 1-22.

\bibitem{GLY}  A. Grigor’yan, Y. Lin, Y. Yang, {\it Kazdan–Warner equation on graph}, Calc. Var. Partial Differential Equations 55 (2016), 1-13.


\bibitem{Ha} X. Han, {\it A Sharp Existence Theorem for Vortices in the Theory of Branes}, Annales Henri Poincaré, 15(12) (2014), 2467-2487.

\bibitem{HS} S. Hou, J. Sun, {\it Existence of solutions to Chern–Simons–Higgs equations on graphs},  Calc. Var. Partial Differential Equations 61 (2022). 

\bibitem{Hu} Y. Hu, {\it Existence of solutions to a generalized self-dual Chern-Simons equation on finite graphs}, J. Korean Math. Soc. 61(1) (2024), 133-147.

\bibitem{Hu1} Y. Hu, {\it Existence and uniqueness of solutions to the Bogomol'nyi equation on graphs}, arXiv: 2202.05039 (2022).

\bibitem{HWY} H.Y. Huang, J. Wang, W. Yang, {\it Mean field equation and relativistic Abelian Chern-Simons model on finite graphs}, Journal of Functional Analysis 281(10) (2021), 109218.

\bibitem{HLY}  A. Huang,  Y. Lin, S. T. Yau, {\it Existence of solutions to mean field equations on graphs}, Comm. Math. Phys. 377 (2019), 613-621.


\bibitem{LP}  Y. Lü, P. Zhong, {\it Existence of solutions to a generalized self-dual Chern-Simons equation on graphs}, arXiv: 2107.12535 (2021).

\bibitem{LY4} C.S. Lin, Y. Yang, {\it Non-Abelian multiple vortices in supersymmetric field theory}, Commun. Math. Phys. 304  (2011), 433–457.

\bibitem{LY5} C.S. Lin, Y. Yang, {\it Sharp existence and uniqueness theorems for non-Abelian multiple vortex solutions},
Nucl. Phys. B 846 (2011), 650–676.

\bibitem{LW} Y. Lin, Y. Wu, {\it Blow-up problems for nonlinear parabolic equations on locally finite graphs}, Acta Mathematica Scientia 38(3) (2018), 843-856.


\bibitem{S} T. Suyama, {\it Intersecting branes and generalized vortices}, arXiv: hep-th/9912261v1.

\bibitem{SY} M. Shifman,  A. Yung, {\it Non-Abelian string junctions as confined monopoles}, Phys. Rev. D 70 (2004), 045004.

\bibitem{SY8} M. Shifman, A.Yung,  {\it Localization of non-Abelian gauge fields on domain walls at weak coupling:
D-brane prototypes}, Phys. Rev. D 70 (2004), 025013.


\bibitem{SY9}  M. Shifman, A.Yung, {\it Supersymmetric solitons and how they help us understand non-Abelian gauge
theories}, Rev. Mod. Phys. 79 (2007), 1139 .

\bibitem{SY1} M. Shifman, A. Yung,  {\it Supersymmetric Solitons}, Cambridge University Press, 2009.

 \bibitem{T1}  C.H. Taubes, {\it Arbitrary N-vortex solutions to the first order Ginzburg–Landau
          equations}, Commun. Math. Phys. 72 (1980), 277–292.
          
 \bibitem{T2} C.H. Taubes, {\it On the equivalence of the first and second order equations for
          gauge theories}, Commun. Math. Phys. 75 (1980), 207–227.

\bibitem{T} G. Tarantello, {\it Multiple condensate solutions for the Chern–Simons–Higgs theory}, J. Math.
Phys. 37 (1996) 3769–3796.

\bibitem{TZZ}  Canrong Tian, Qunying Zhang, Lai Zhang, {\it Global stability in a networked SIR epidemic model}, Appl. Math. Lett. 107 (2020), 106444.

\bibitem{TY} D. H. Tchrakian, Y. Yang, {\it The existence of generalised self-dual Chern-Simons vortices},
Lett. Math. Phys. 36(4)  (1996), 403-413.

\bibitem{WC}  Y. Wu, {\it On-diagonal lower estimate of heat kernels for locally finite graphs and its application to the semilinear heat equations}, Comput. Math. Appl. 76 (2018), 810-817.

\bibitem{WN}  Y. Wu, {\it On nonexistence of global solutions for a semilinear heat equation on graphs}, Nonlinear Analysis 171 (2018), 73-84.


\bibitem{Y} Y. Yang, {\it Chern-Simons solitons and a nonlinear elliptic equation}, Helv. Phys. Acta 71(5) (1998), 573-585.

\bibitem{LY} E. H. Lieb, Y. Yang,  {\it Non-Abelian vortices in supersymmetric gauge field theory via direct methods}, Communications in Mathematical Physics 313(2) 2012, 445-478.

\bibitem{JT} A. Jaffe, C. H. Taubes, {\it Vortices and monopoles}, Birkhäuser, Boston (1980).




\end{thebibliography}
\end{document}